\documentclass[12pt]{amsart} 
\usepackage[utf8]{inputenc}
\usepackage{geometry}
\usepackage{mathrsfs}
\usepackage{amssymb,amsmath,amsthm,color}
\usepackage{graphicx}
\usepackage[colorlinks=true,linkcolor=blue,citecolor=blue]{hyperref}
\usepackage{hyperref}
\usepackage{url} 
\usepackage{setspace}
\usepackage{mathtools}
\usepackage[inline]{enumitem}  
\usepackage{cite}
\numberwithin{equation}{section}
\newtheorem{theorem}{Theorem}[section]

\newtheorem{Remark}{Remark}[section]
\newtheorem{proposition}{Proposition}[section]
\newtheorem{Lemma}{Lemma}[section]

\newcommand{\R}{\mathbb R}

\def\rdd{\bR^{2d}}
\usepackage{color}

\def\bR{\mathbb{R}}
\def\bC{\mathbb{C}}
\def\bN{\mathbb{N}}

\def\cS{\mathcal{S}}

\def\rd{\bR^d}

\def\rdd{\bR^{2d}}

\def\la{\langle}
\def\ra{\rangle}

\def\*b{*_{\bullet}}

\def\S0{S^0_{0,0}}

\def\Bd'{B_{\delta'}}

\def\cBd'{\bar{B}_{\delta'}}

\begin{document}
\baselineskip16pt
\title[]{On heat equations associated with fractional harmonic oscillators}
\author[D. Bhimani]{Divyang G. Bhimani}
\address{Department of Mathematics, Indian Institute of Science Education and Research, Dr. Homi Bhabha Road, Pune 411008, India}
\email{divyang.bhimani@iiserpune.ac.in}
\author[R.  Manna]{Ramesh Manna}
\address{School of Mathematical Sciences, National Institute of Science Education and Research Bhubaneswar, An OCC of Homi Bhabha National Institute, Jatni 752050, India}
\email{rameshmanna@niser.ac.in}
\author[F. Nicola ]{Fabio Nicola }
\address{Dipartimento di Scienze Matematiche ``G. L. Lagrange'', Politecnico di Torino, corso Duca degli Abruzzi 24, 10129 Torino, Italy}
\email{fabio.nicola@polito.it}
\author[S.  Thangavelu]{Sundaram Thangavelu}
\address{Department of Mathematics, Indian Institute of Science, 560 012 Bangalore, India}
\email{veluma@iisc.ac.in}
\author[S.I. Trapasso]{S. Ivan Trapasso}
\address{Dipartimento di Scienze Matematiche ``G. L. Lagrange'', Politecnico di Torino, corso Duca degli Abruzzi 24, 10129 Torino, Italy}
\email{salvatore.trapasso@polito.it}

\date{}

\maketitle
 \begin{abstract} We establish some fixed-time decay estimates in Lebesgue spaces for the fractional heat propagator $e^{-tH^{\beta}}$, $t, \beta>0$, associated with the harmonic oscillator $H=-\Delta + |x|^2$. We then prove some local and global wellposedness results for nonlinear fractional heat equations.  
  \end{abstract}
\section{Introduction}\label{id}
Consider the heat equation associated with the fractional harmonic oscillator, namely
\begin{equation}\label{fHTEhom}
\begin{cases}
\partial_t u(t,x) + H^{\beta}u(t,x)=0\\
u(0,x)= u_0(x),
\end{cases}  \quad (t, x) \in \R^{+}\times  \R^d, 
\end{equation}
where $H^{\beta}=(-\Delta +|x|^2)^{\beta}$, $\beta>0$, and $u(t,x)\in \bC$. 

Strictly speaking, the corresponding fractional heat semigroup $e^{-t H^{\beta}}$ is defined in terms of the spectral decomposition of the standard Hermite operator $H=H^{1}=-\Delta+|x|^2$. To be precise, recall that \[ H = \sum_{k=0}^\infty (2k+d) P_k, \] where $P_k$ stands for the orthogonal projection of $ L^2(\mathbb R^d) $ onto the eigenspace corresponding to the eigenvalue $(2k+d)$ -- see Section \ref{fhm} below for further details. As a consequence of the spectral theorem, we can consider the family of fractional powers of $H$ defined by
\[ H^\beta = \sum_{k=0}^\infty (2k+d)^\beta P_k, \quad \beta >0. \] The heat semigroup $e^{-t H^{\beta}}$ is then defined accordingly by 
\[e^{-t H^{\beta}}f =   \sum_{k=0}^\infty e^{-t(2k+d)^{\beta}} P_kf, \quad f \in L^2(\rd). \]

While there is a wealth of literature on the semigroup $e^{-t (-\Delta)^\beta}$ (see e.g., \cite{miao2008well, weissler1980local}), stimulated by the very wide range of physics-inspired models involving the fractional Laplacian \cite{laskin2002fractional, garofalo2017fractional}, the current research of the semigroup $e^{-tH^\beta}$ is rather limited, even in fundamental settings such as the Lebesgue spaces. This is particularly striking in view of the role played by the Hermite operator $H$ and its fractional powers $H^\beta$ in several aspects of quantum physics and mathematical analysis \cite{thangavelu1993lectures, lee2012schrodinger}. 

The purpose of this note is to advance the knowledge of the fractional heat semigroup, in the wake of a research program initiated by the authors in \cite{bhimani2020phase}. In particular, our main result is a set of fixed-time decay estimates for $e^{-tH^\beta}$ in the Lebesgue space setting. 

\begin{theorem}\label{mt} For $1 \leq  p,q\leq\infty$ and $\beta >0,$ set
\[
 \quad \sigma_\beta \coloneqq  \frac{d}{2\beta} \Big|\frac{1}{p}-\frac{1}{q}\Big|.
 \]
\begin{enumerate}
\item \label{P1} 
 If  $p,q \in (1,  \infty),$ or $(p,q)=(1, \infty),$ or $p=1$ and $q \in [2, \infty),$ or  $p\in (1, \infty)$ and $q=1,$ then there exists a constant $C>0$ such that
\begin{equation}\label{eq mainthm}
\|e^{-tH^\beta} f\|_{L^q} \le  \begin{cases}
C e^{-td^\beta} \|f\|_{L^p}  &  \text{if} \quad t\geq 1\\
 C t^{-\sigma_\beta} \|f\|_{L^p} &  \text{if} \quad   0<t\leq 1.
\end{cases} 
\end{equation} 
\item \label{P2} 
If $0<\beta \leq 1,$ then the above estimate holds for $p,q \in [1,  \infty].$
\end{enumerate}
\end{theorem}

To the best of our knowledge, the dissipative estimate in Theorem \ref{mt} is new even for the Hermite operator ($\beta=1$). We also stress that the time decay at infinity in \eqref{eq mainthm} is sharp for any choice of  Lebesgue exponents. Moreover, since the power of $t$ is never positive for small time, we infer that there is a singularity near the origin for $p\neq q$.  

It is worth emphasizing that the fractional Hermite propagator $e^{-tH^\beta}$ is not a Fourier multiplier, hence we cannot rely on the arguments typically used to establish $L^p-L^q$ space-time estimates for the fractional heat propagator $e^{-t(-\Delta)^\beta}$ -- see for instance \cite[Lemma 3.1]{miao2008well}. In fact, we will resort to techniques of pseudodifferential calculus to deal with the operators $e^{-tH^\beta}$ and $e^{-tH}$ (cf.\ \cite[Section 4.5]{NicolaBook}), and also to Bochner's subordination formula in order to express the heat semigroup $e^{-tH^{\beta}}$, $0<\beta\leq 1$, in terms of solutions of the heat equation $ e^{-tH}$ (see \eqref{ifh}).

As an application of Theorem \ref{mt}, we investigate the wellposedness of 
\begin{equation}\label{fHTE}
\begin{cases}
\partial_t u(t,x) + H^{\beta}u(t,x)= |u(t,x)|^{\gamma-1} u(t,x)\\
u(0,x)= u_0(x),
\end{cases}  \quad (t, x) \in \R^{+}\times  \R^d, 
\end{equation} with $u(t,x)\in \bC$, $\beta>0$ and $\gamma>1$.

First, let us highlight that, due to the occurrence of the quadratic potential $|x|^2$, the problem \eqref{fHTE} has no scaling symmetry. Nevertheless, the companion fractional heat equation
\begin{equation}\label{fHTEs}
\begin{cases}
\partial_t u(t,x) + (-\Delta)^{\beta}u(t,x)= |u(t,x)|^{\gamma-1}u(t,x)\\
u(0,x)= u_0(x),
\end{cases}  \quad \quad (t, x) \in \R^{+}\times  \R^d ,
\end{equation}
is invariant under the following scaling transformation. For $\lambda>0$, set 
\[ u_{\lambda} (t,x) = \lambda^{\frac{2\beta}{\gamma -1}} u(\lambda^{2\beta} t, \lambda x) \quad \text{and} \quad  u_{0, \lambda}(x) = \lambda^{\frac{2\beta}{\gamma -1}} u_0(\lambda x). \]
If $u(t,x)$ is a solution of \eqref{fHTEs}  with  initial datum $u_0(x)$, then $u_{\lambda}(t,x)$ is also a solution of  \eqref{fHTEs} with initial datum  $u_{0, \lambda}(x)$.  The $L^p$ space is invariant under the above scaling only when $p=p^\beta_c \coloneqq \frac{d (\gamma-1)}{2\beta}$. Motivated by this remark, we shall say that \eqref{fHTE} is 
\begin{equation*}L^p-
\begin{cases} \text{sub-critical} & \text{if} \quad  1\leq p <p^\beta_c\\
\text{critical}  & \text{if} \quad p=p^\beta_c\\
\text{super-critical}  & \text{if} \quad p>p^\beta_c.
\end{cases}
\end{equation*}

Concerning the wellposedness of \eqref{fHTE}, our result can be stated as follow. 
\begin{theorem}\label{LWL} Assume that  $u_0 \in L^p(\R^d), 1<p< \infty$ and $\beta>0.$
\begin{enumerate}
	\item (Local well-posedness) \label{LWL1} If  $p>p_c^{\beta},$ then there
exists a $T>0$  such that \eqref{fHTE} has a  solution  $u \in C([0, T ], L^p (\R^d)).$  Moreover, $u$ extends to  a maximal interval $[0, T_{\max})$ such that either $T_{\max}=\infty$ or $T_{\max} < \infty$ and $\displaystyle \lim_{t\to T_{\max}} \|u(t)\|_{L^p} = \infty.$
\item \label{LWL2} (Lower blow-up rate) Consider $p>p_c^{\beta}$ and suppose that $T_{\max}<\infty,$  where $T_{\max}$ is the existence time of the resulting maximal solution of \eqref{fHTE}. Then
\[ \|u(t)\|_{L^p} \geq C  \left( T_{\max}- t \right) ^{\frac{d}{2p\beta}-\frac{1}{\gamma-1}}, \quad \text{for all} \ t \in [0, T_{\max}). \]
\item (Global existence)  \label{T1.3} If  $p=p_c^{\beta}$ and $\|u_0\|_{L^{p_c^{\beta}}}$ is sufficiently small, then $T_{max}=\infty.$
\end{enumerate}
\end{theorem}

Let us briefly recall the literature to better frame our results. Weissler \cite{weissler1980local} proved local wellposeness for \eqref{fHTEs} in $L^p$ for super-critical indices $p> p^1_c\geq 1$. Concerning the sub-critical regime $p< p_c^1$, there is no general theory of existence, see \cite{weissler1980local, brezis1996nonlinear}. Actually, Haraux-Weissler \cite{haraux1982non} proved that if  $1<p^1_c <  \gamma +1$ then there is a global solution of \eqref{fHTEs} (with zero initial data) in $L^p(\R^d)$ for $1\leq p < p^1_c$, but no such solution exists when $\gamma +1< p_c$. In the critical case where $p=p^1_c$ it is proved that the solution exists globally in time for small initial data. Some results in the same vein have been proved for the fractional    heat equation \eqref{fHTEs} by Miao, Yuan and Zhang in \cite[Theorem 4.1]{miao2008well}. 

\begin{Remark} Let us discuss some aspects of the previous results. In particular, we highlight some intriguing related problems that we plan to explore in future work. 
\begin{itemize}
    \item[-] The sign in power type non-linearity (focusing or defocusing) will not play any role in our analysis. Therefore, we have chosen to consider the defocusing case for the sake of concreteness.  
    
    \item[-] Using properties of Hermite functions and interpolation, in \cite[Theorem 1.6]{wong2005} Wong proved that  $\|e^{-tH}f\|_{L^2(\R)} \lesssim (\sinh t)^{-1}\|f\|_{L^p(\R)}$ for $t>0$ and $1\leq p \leq 2$. We note that Theorem \ref{mt} recaptures and improves Wong's result. 
    
    \item[-] It is known that \eqref{fHTEs} is ill-posed on Lebesgue spaces in the sub-critical regime,  see  \cite{haraux1982non}. There is reason to believe that the same conclusion holds for \eqref{fHTE}. However, a thorough analysis of this problem is beyond the scope of this note.
    
    \item[-] It is expected that Theorem \ref{mt} could be useful in dealing with other types of non-linearities in \eqref{fHTE}, such as exponential and inhomogeneous type non-linearity (which are also extensively studied in the literature). 
  
    \item[-] In Section \ref{cr} we discuss another application of Theorem \ref{mt}, namely Strichartz estimates for the fractional heat semigroup. Our approach here relies on a standard technique (i.e., $TT^{\star}$ method and  real  interpolation), whereas a refined phase-space analysis of $H^\beta$ is expected to reflect into better estimates.  
      \end{itemize}
\end{Remark}


\section{Preliminary results}\label{pki}

\noindent \textbf{Notation.} The symbol $X \lesssim Y$ means that the underlying inequality holds with a suitable positive constant factor: 
\begin{equation*} X \lesssim Y \quad\Longrightarrow\quad\exists\, C>0\,:\,X \le C Y. \end{equation*}

\subsection{On the fractional harmonic oscillator $H^{\beta}$}\label{fhm} Let us briefly review some facts concerning the spectral decomposition of the Hermite operator $H=-\Delta + |x|^2$ on $\rd$. 

Let $\Phi_{\alpha}(x)$, $\alpha \in \mathbb N^d$, be the normalized $d$-dimensional Hermite functions, that is
\[ \Phi_\alpha(x) = \Pi_{j=1}^d  h_{\alpha_j}(x_j), \quad h_k(x) = (\sqrt{\pi}2^k k!)^{-1/2} (-1)^k e^{\frac{1}{2}x^2}  \frac{d^k}{dx^k} e^{-x^2}.\]
The Hermite functions $ \Phi_\alpha $ are eigenfunctions of $H$ with eigenvalues $(2|\alpha| + d)$,  where $|\alpha |= \alpha_{1}+ ...+ \alpha_d$. Moreover, they form an orthonormal basis of $ L^2(\R^d)$. The spectral decomposition of $ H $ is thus given by
\[ H = \sum_{k=0}^\infty (2k+d) P_k, \qquad P_kf = \sum_{|\alpha|=k} \langle f,\Phi_\alpha\rangle \Phi_\alpha, \] 
where $\langle\cdot, \cdot \rangle $ is the inner product in $L^2(\mathbb{R}^d)$. 

In general, given a bounded function $m \colon \bN \to \bC$, the spectral theorem allows us to define the operator $m(H)$ such that
\[ m(H)f= \sum_{\alpha \in \mathbb N^d} m(2|\alpha| +d) \langle f, \Phi_{\alpha} \rangle \Phi_{\alpha} = \sum_{k=0}^\infty m(2k+d)P_kf, \quad f \in L^2(\rd).\] 
In view of the Plancherel theorem for the Hermite expansions, $m(H)$ is bounded on $L^{2}(\mathbb R^d)$. We refer to \cite{thangavelu1993lectures} for further details, in particular for H\"ormander multiplier-type results for $m(H)$ on $L^p(\R^d)$.
\subsection{Some relevant function spaces}
For the benefit of the reader we review some basic facts of time-frequency analysis -- see for instance \cite{grochenig2013foundations,cordero2020time, KassoBook} for comprehensive treatments. 

Recall that the short-time Fourier transform  of a temperate distribution $f \in \mathcal{S}'(\R^d)$ with respect to a window function $0\neq g \in {\mathcal S}(\R^d)$ (Schwartz space) is defined by
\begin{eqnarray*}\label{stft}
V_{g}f(x,\xi)= \la f,g \ra = \int_{\mathbb R^{d}} f(t) \overline{g(t-x)} e^{-2\pi i \xi \cdot t}dt,  \  (x, \xi) \in \mathbb R^{2d},
\end{eqnarray*} where the brackets $\la \cdot,\cdot \ra$ denote the extension to $\cS'(\rd)\times \cS(\rd)$ of the $L^2$ inner product. 

Modulation spaces, introduced by Feichtinger \cite{Feih83}, have proved to be extremely useful in a wide variety of contexts, ranging from analysis of PDEs to mathematical physics -- among the most recent contributions, see e.g., \cite{dias_22, manna_22, bhimani_19, NT_jmp,fei_21}. Modulation spaces are defined as follows. For $1 \leq p,q \leq \infty$  we have 
 \[ M^{p,q}(\R^d)= \left\{ f \in \mathcal{S}'(\R^d): \|f\|_{M^{p,q}} \coloneqq \left\| \|V_gf(x,\xi)\|_{L^p_x}  \right\|_{L_\xi^q}< \infty\right\} . \]

We recall from \cite[Theorem 1.1]{bhimani2020phase} some bounds for the fractional heat semigroup on modulation spaces. 
\begin{theorem} \label{BMNTT} Let $\beta>0$, $0< p_1,p_2,q_1,q_2\leq\infty$ and set
\[
 \frac{1}{\tilde{p}}\coloneqq \max\Big\{\frac{1}{p_2}-\frac{1}{p_1},0\Big\},\ \ \frac{1}{\tilde{q}}\coloneqq \max\Big\{\frac{1}{q_2}-\frac{1}{q_1},0\Big\},\quad \sigma_\beta \coloneqq  \frac{d}{2\beta} \Big(\frac{1}{\tilde{p}}+\frac{1}{\tilde{q}}\Big).
 \] Then
\begin{equation*}
\|e^{-tH^\beta} f\|_{M^{p_2,q_2}} \le \begin{cases}
C e^{-td^\beta} \|f\|_{M^{p_1,q_1}} & \text{if} \quad t\geq 1\\
C t^{-\sigma_\beta} \|f\|_{M^{p_1,q_1}} & \text{if} \quad 0<t\leq 1, 
\end{cases}
\end{equation*} where $C>0$ is a universal constant.
\end{theorem}

We briefly recall some properties of the Shubin classes $\Gamma^s$, which play a central role as symbol classes in the theory of pseudodifferential operators -- we refer to \cite{NicolaBook} for additional details. For $s\in\mathbb{R}$ we define $\Gamma^s$ as the space of functions $a\in C^\infty(\rdd)$ satisfying the following condition: for every $\tilde{\alpha} \in \mathbb{N}^{2d}$ there exists $C_{\tilde{\alpha}}>0$ such that
\[
|\partial^{\tilde{\alpha}} a(z)|\leq C_{\tilde{\alpha}} (1+|z|)^{s-|\tilde{\alpha}|},\qquad z\in\rdd,
\]
This space becomes a Fr\'echet space endowed with the obvious seminorms. 

It is important for our purposes to recall that the fractional Hermite propagator is a pseudodifferential operator with symbol in a suitable Shubin class, as proved in \cite[Proposition 2.3]{bhimani2020phase}.
\begin{proposition} \label{prop shubin}
	Let $\beta >0$. The fractional Hermite operator $H^{\beta} = (-\Delta+|x|^2)^{\beta}$ is a pseudodifferential operator with Weyl symbol $a_{\beta} \in \Gamma^{2\beta}$. More precisely, we have
	\begin{equation}\label{eq prop shubin}
	a_\beta(x,\xi) = (|x|^2+|\xi|^2)^\beta + r(x,\xi), \quad |x|+|\xi|\ge 1,
	\end{equation}  where $r \in \Gamma^{2\beta-2}$.
\end{proposition}

We also recall some facts concerning the so-called Shubin-Sobolev (also known as Hermite-Sobolev) spaces $Q^s$, $s\in\mathbb{R}$ -- see \cite{shubin1987pseudodifferential}, \cite[Theorem \ 2.1]{todor} for further details. In particular, $Q^s$ is the space of $f\in\cS'(\rd)$ such that 
\begin{equation*}\label{eq char shusob}
\|f\|^2_{Q^s}\coloneqq \|H^{s/2}f\|^2_{L^2}=\sum_{k=0}^{\infty} ||P_k f||^2_{L^2}(2k+d)^{s}<\infty.
\end{equation*}
In view of the characterisation $Q^s=M^{2,2}_{v_s}$ (see for instance \cite[Lemma 4.4.19]{cordero2020time}), H\"older's inequality and the inclusion relations of Shubin-Sobolev spaces (see e.g., \cite[Theorem 2.4.17]{cordero2020time}), it is well known that, for every $1\leq p,q\leq\infty$, if $s$ is large enough,
\begin{equation*}\label{eq emb shusob}
Q^s\hookrightarrow M^{p,q}\hookrightarrow M^\infty\hookrightarrow Q^{-s}.
\end{equation*}  

\section{Proof of Theorem \ref{mt}} \label{s3}
\subsection{Proof of Part \eqref{P1}}
It is known that 
\[
L^p \hookrightarrow M^{p,\infty} \text{ and } M^{q,1} \hookrightarrow L^q \text{ for } 1\leq p,q \leq \infty,
\] see e.g., \cite{cordero2020time,sugimoto2011remarks}. In light of this embedding and Theorem \ref{BMNTT},  for $t>1$ we obtain the desired estimate
\[ \|e^{-tH^\beta} f\|_{L^q}\lesssim e^{-td^\beta}\|f\|_{L^p}, \quad \forall p, q \in [1, \infty]. \]

Let us consider now the case where  $0<t \leq 1$. In view of Proposition \ref{prop shubin} we think of $H^{\beta}$ as a pseudodifferential operator with Weyl symbol $a_{\beta} \in \Gamma^{2\beta}$, where
	\begin{equation*}
	a_\beta(x,\xi) = (|x|^2+|\xi|^2)^\beta + r(x,\xi), \quad |x|+|\xi|\ge 1,
	\end{equation*}  
	for a suitable $r \in \Gamma^{2\beta-2}$. We may further rewrite 
	\begin{equation*}
	a_\beta(x,\xi) = a(x, \xi) + r'(x,\xi), \quad x, \xi \in \mathbb{R}^d,
	\end{equation*} 
	for some $r' \in \Gamma^{2\beta-2}$, where $a\in \Gamma^{2\beta}$ satisfies
	\begin{align} \label{modifiedsym}
	a(x,\xi) \ge (1+|x|+|\xi|)^{2\beta}, \quad x, \xi \in \mathbb{R}^d.
	\end{align}
Note that the same conclusion holds for the Kohn-Nirenberg symbol of $H^{\beta}$ (see \cite[Proposition 1.2.9]{NicolaBook}). Therefore, we assume hereinafter that the above functions $a(x,\xi),~r'(x,\xi)$ denote the Kohn-Nirenberg symbol of the corresponding operators.

It follows from \cite[Theorem 4.5.1]{NicolaBook} that the heat semigroup $e^{-tH^\beta}$ has a Kohn-Nirenberg symbol with the following structure\footnote{Note that the mentioned result is stated for the Weyl quantization, but again one can easily check that the same conclusion holds for the Kohn-Nirenberg quantization.}: 
\[
b_t(x,\xi)=e^{-t a(x,\xi)}+e^{-t a(x,\xi)} \sum_{j=1}^{J-1} \sum_{l=1}^{2j} t^l u_{l,j}(x,\xi) + r_t^{''}(x,\xi),
\]
where $J\ge 1$ is arbitrarily chosen, $u_{l,j} \in \Gamma^{2\beta l-2j}$ and $r_t^{''}$ satisfy
\[\big|\partial_x^{\alpha} \partial_{\xi}^{\gamma}r_t^{''}(x,\xi)\big| \leq C_{\alpha, \gamma} (1+|x|+|\xi|)^{-2J-|\alpha|-|\gamma|}\]
for a constant $C_{\alpha, \gamma} $ independent of $t \in (0,1),$ for  every $\alpha, \gamma \in \mathbb{N}^{d}$.

Since $r_t^{''}(x, D)  \colon Q^{-J} \to Q^J,$  for $J$ large enough,  we have 
\[\|r_t^{''}(x, D) f\|_{L^q} \leq C \|f\|_{L^p}.\]

Let us focus now on the symbol 
\[C_t(x,\xi)\coloneqq  e^{-t a(x,\xi)} \sum_{j=1}^{J-1} \sum_{l=1}^{2j} t^l u_{l,j}(x,\xi).\]
By virtue of the Leibniz rule, the chain rule and  \eqref{modifiedsym}, one can verify the estimates 
\begin{align} \label{symbolcal}
\big|\partial_x^{\alpha} \partial_{\xi}^{\gamma}[e^{\frac t4 \langle x \rangle ^{2\beta}} C_t(x,\xi)]\big| \leq C_{\alpha, \gamma} (1+|\xi|)^{-|\gamma|},
\end{align}
where $\langle \cdot \rangle = (1+ |\cdot|^2)^{1/2}$. In fact, it suffices to observe that $\partial_x^{\alpha}e^{\frac t4 \langle x\rangle^{2\beta}}$ is a finite linear combination of terms of the type 
\[
e^{\frac t4 \langle x\rangle^{2\beta}}~~\partial^{\alpha_1} [t \langle x\rangle^{2\beta}] \cdots \partial^{\alpha_k} [t \langle x\rangle^{2\beta}],
\]
with $\alpha_1 +\cdots+\alpha_k=|\alpha|$, so that 
\[
\big|\partial_x^{\alpha}e^{\frac t4 \langle x\rangle^{2\beta}}\big| \leq e^{\frac t42\langle x\rangle^{2\beta}} ~\langle x\rangle^{-|\alpha|}. 
\]
Similarly, since $a \in \Gamma^{2\beta}$ satisfies \eqref{modifiedsym}, we have
\[
\big|\partial_x^{\alpha} \partial_{\xi}^{\gamma}a(x,\xi)\big| \leq a(x,\xi) \,  (1+|x|+|\xi|)^{-|\alpha|-|\gamma|},
\]
so that, arguing as above,
\[
\big|\partial_x^{\alpha} \partial_{\xi}^{\gamma}e^{-t a(x,\xi)}\big| \leq e^{-\frac t2 a(x,\xi)} \,  (1+|x|+|\xi|)^{-|\alpha|-|\gamma|},
\]
hence we infer
\[
\big|\partial_x^{\alpha} \partial_{\xi}^{\gamma}[t^l \, u_{l,j}(x,\xi)]\big| \leq  t^l \,  a(x,\xi)^l \,  (1+|x|+|\xi|)^{-2j-|\alpha|-|\gamma|}.
\]
The claimed bound thus follows by the Leibniz rule.

To summarize, for every $p \in (1,\infty)$ we have
\[\|e^{\frac t4 \langle x\rangle^{2\beta}} C_t(x,D)f\|_{L^p} \leq \|f\|_{L^p},\quad 0<t<1,\]
by the $L^p$ boundedness of pseudodifferential operators with symbol in H\"{o}rmander's class $S_{1,0}^0$ -- see for instance \cite[Proposition 4, p. 250]{stein1993harmonic}. 
For  $1\leq q \leq p \leq \infty$ we have, by H\"{o}lder inequality,
\[
\|e^{-\frac t4 \langle x\rangle^{2\beta}} f\|_{L^q} \leq C t^{ \frac{d}{2\beta} \left(\frac{1}{q}-\frac{1}{p}\right)} \|f\|_{L^p}.
\]
Hence we obtain, for $1\leq q \leq \infty,~1<p<\infty,~q\le p$,
\[
\|C_t(x, D) f\|_{L^q} \leq C t^{ \frac{d}{2\beta} \left(\frac{1}{q}-\frac{1}{p}\right)} \|f\|_{L^p}, \quad 0<t<1.
\]

On the other hand, we also have 
\begin{equation*} \label{zerosymb}
\big|C_t(x,\xi)\big| \leq C e^{-\frac t2 |\xi|^{2\beta}},\quad 0<t<1,
\end{equation*}
and the integral kernel of the operator $C_t(x,D)$ given by
\[
K(x,y)=(2 \pi)^{-d} \int_{\mathbb{R}^d} e^{i(x-y) \cdot \xi} C_t(x,\xi) \, d\xi
\]
is readily seen to satisfy 
\[
\big|K(x,y)\big| \leq C t^{- \frac{d}{2\beta}}.
\]
This gives the desired continuity result $L^1 \to L^{\infty}$, while the remaining bounds follow by interpolation with the above  $L^p \to L^q$ estimates.

\begin{Remark}
Note that some endpoint cases can be obtained in a straightforward way. For instance, from $L^1 \to L^{\infty}$ continuity we also obtain $L^1 \to L^2$ bounds as follows: if $f\in L^2(\rd)$ then
\[
\big|\langle e^{-t H^{\beta}}f, e^{-t H^{\beta}}f \rangle  \big|=\big|\langle e^{-2t H^{\beta}}f, f \rangle  \big| \le C t^{-\frac{d}{2\beta}} \|f\|_{L^1}^2
\]
so that
\[
\|e^{-t H^{\beta}}f\|_{L^2} \leq C t^{-\frac{d}{4\beta}} \|f\|_{L^1}, \quad 0<t<1.
\]
By interpolation with $L^1 \to L^{\infty}$ one also gets the desired estimate $L^1 \to L^q$ for $2\leq q \leq \infty$.
\end{Remark}

\begin{Remark}
Some endpoint cases (e.g., if $p,q \in \{1,\infty\}$) are not covered in the results above. A deeper investigation of the kernel $K(x,y)$ of $C_t(x,D)$ could likely give some result in this connection (for example $L^1 \to L^1,~~L^{\infty} \to L^{\infty}$), but it will not be essential for the applications to the nonlinear problem in Theorem \ref{LWL}. Nevertheless, the dispersive estimate $L^1 \to L^{\infty}$ is covered.
\end{Remark}

\subsection{Proof of Part \eqref{P2}}
 In order to prove the second claim in Theorem \ref{mt}, some preparatory work is needed. First, we recast $e^{-tH}$ as the Weyl transform of a function on $\mathbb{C}^d$, which allows us to think of $e^{-tH}$ as a pseudodifferential operator. 
 
 Recall that the Weyl transform $W(F)$ of a function $F \colon \bC^d \to \bC$ is defined by
\[ W(F)\phi(\xi)=(2\pi)^{-d} \int_{\R^d}\int_{\R^d} e^{i(\xi-\eta)\cdot y} b\left(\frac{\xi+\eta}{2},y\right) \, \phi(\eta) \, dy d\eta, \]
 for $\phi \in L^2(\R^d)$, where the symbol $b(\xi, \eta)$ is the full inverse Fourier transform of $F$ in both variables. In particular, the Weyl transform $W(F)$ is a pseudodifferential operator in the Weyl calculus with symbol $b$. 
 
Let us highlight that the Weyl symbol of the Hermite semigroup $e^{-tH}$ is given by the function $a_t(x, \xi) =C_d(\cosh t)^{-d} \, e^{-(\tanh t)(|x|^2+|\xi|^2)}$, see \cite{thangavelu2018note}.  Thus,
 \begin{equation*} 
 e^{-tH}f(x) = C_d (\cosh t)^{-d}(2\pi)^{-d} \underset{=I}{\underbrace{\int_{\R^d}\int_{\R^d} e^{i(x-\eta)\cdot y} \, e^{-(\tanh t)|y|^2} \, e^{-(\tanh t)(|\frac{x+\eta}{2}|^2)} \, f(\eta) \, dy d\eta}}.
 \end{equation*}
In order to  bound the above integral $I$, we first recast the latter expression in terms of convolution. Recall that the Fourier transform of the Gaussian function $f(y)= e^{-\pi a |y|^2}$ with $a>0$ is given by $\hat{f}(x)= a^{-d/2} e^{-\pi |x|^2/a}$, and  note that 
\[\frac{|x-\eta|^2}{4} - \frac{|x|^2}{2} - \frac{|\eta|^2}{2} = - \frac{|x + \eta|^2}{4}. \]
As a result, we have 
\begin{eqnarray*}
\left(\tanh t \right) ^{d/2} I 
 =    \int_{\rd} \, e^{-(\frac{1}{4  \tanh t}-\frac{\tanh t}{4}) |x- \eta|^2} e^{- \frac{\tanh t}{2} (|x|^2 + |\eta|^2) } f(\eta) \, d\eta
 =   e^{-\frac{\tanh t}{2} |x|^2}  \left( e^{- \frac{1}{2\sinh 2t} |\cdot |^2} \ast g\right) (x),
\end{eqnarray*}
where we set $g(\cdot)= e^{-\frac{\tanh t}{2}|\cdot|^2} \, f(\cdot)$.  Note that $(\cosh t)^{-d} \left(\tanh t \right) ^{-d/2}=  \left( \sinh (2t)  \right)^{-d/2}$, hence 
\begin{eqnarray}\label{re}
e^{-tH}f(x)= \tilde{C}_d \left( \sinh (2t)  \right)^{-d/2} e^{-\frac{\tanh t}{2} |x|^2}  \left( e^{- \frac{1}{2\sinh 2t} |\cdot |^2} \ast g\right) (x).
\end{eqnarray}

\begin{Lemma} \label{dc} Let $1\leq  p,q  \leq \infty$ and $t>0$. Then 
\[ \|e^{-tH} f\|_{L^q} \leq C   (\tanh t)^{-\frac{d}{2} \left|\frac{1}{q} - \frac{1}{p}\right| } \, \|f\|_{L^{p}},\] for some constant $C>0$ that depends only on $d$. 
\end{Lemma}
\begin{proof}
 Using Mehler's formula for the Hermite functions  (see e.g., \cite{thangavelu1993lectures}), the kernel $ K_t(x,y) $ of the semigroup $ e^{-tH} $ is explicitly given by 
\[ K_t(x,y) = c_d (\sinh 2t)^{-d/2}  e^{-\frac{1}{4} (\coth t) |x-y|^2} e^{-\frac{1}{4} (\tanh t)|x+y|^2}. \]
For $ 1 < p < q< \infty, $ set $ \alpha = d (1/p-1/q)$. Then we have
\begin{multline*}
 K_t(x,y) \\
 =  c_d (\sinh 2t)^{-d/2} (\tanh t)^{(d-\alpha)/2} |x-y|^{\alpha-d}  ( (\coth t)|x-y|^2)^{(d-\alpha)/2}  e^{-\frac{1}{4} (\coth t) |x-y|^2} e^{-\frac{1}{4} (\tanh t)|x+y|^2},
\end{multline*}
from which we obtain the estimate  
\[ K_t(x,y) \leq C (\cosh t)^{-d}  (\tanh t)^{-\alpha/2} |x-y|^{\alpha-d}. \]
Since the Riesz potential 
$$ R_\alpha f(x) = c_\alpha \int_{\R^d} f(y) |x-y|^{\alpha-d} dy$$ 
is bounded from $ L^p $ to $ L^q $ for $1<p<q<\infty$, we get 
\[ \| e^{-tH} f\|_{L^q} \leq C (\cosh t)^{-d}  (\tanh t)^{-\alpha/2} \|f\|_{L^p} \] for $1 < p < q< \infty$. 

To prove the remaining cases, we use the identity \eqref{re}. We consider the case $ 1\leq q\leq p\leq \infty$ first. Set $\frac{1}{q} = \frac{1}{p} + \frac{1}{\tilde{q}} $ and note that 
\[ \|e^{- \frac{\tanh t}{2} |\cdot  |^2}\|_{L^{\tilde{q}}}  \sim   (\tanh t)^{-d/ 2\tilde{q}}= (\tanh t)^{\frac{d}{2} \left(\frac{1}{p} - \frac{1}{q} \right)}.\]
By \eqref{re} and invoking H\"{o}lder  and  Young's  inequalities,  we obtain
\begin{align*}
	\|e^{-tH}f\|_{L^{q}} & \lesssim   (\sinh 2t)^{-d/2} \, \|e^{- \frac{\tanh t}{2} |\cdot |^2}\|_{L^{\tilde{q}}} \, \|e^{- \frac{1}{2\sinh 2t} |\cdot |^2} \ast g \|_{L^{p}}\\
	& \lesssim (\sinh 2t)^{-d/2} \, (\tanh t)^{\frac{d}{2} \left(\frac{1}{p} - \frac{1}{q} \right)} \, \|e^{- \frac{1}{2\sinh 2t} |\cdot |^2}\|_{L^1} \| g \|_{L^p}\\
	& \lesssim (\tanh t)^{-\frac{d}{2} \left( \frac{1}{q} - \frac{1}{p} \right)} \| f \|_{L^p}.
\end{align*}
Let $1\leq q\leq \infty$ and note that 
\[ \|e^{- \frac{1}{2\sinh 2t} |\cdot |^2}\|_{L^{q}}  \approx   (\sinh (2t))^{d/ 2q}.\]
By \eqref{re} and Young inequality, we have 
\begin{align*}
\|e^{-tH}f\|_{L^{q}} & \lesssim (\sinh 2t)^{-d/2}  \|e^{- \frac{1}{2\sinh 2t} |\cdot |^2} \ast g \|_{L^{q}}\\
& \lesssim (\sinh 2t)^{-d/2} \|e^{- \frac{1}{2\sinh 2t} |\cdot |^2}\|_{L^{q}} \| g \|_{L^1}\\
& \lesssim (\sinh 2t)^{-\frac{d}{2} \left( 1- \frac{1}{q} \right)} \| f \|_{L^1}\\
& \lesssim (\cosh t)^{-d \left( 1- \frac{1}{q} \right)} \, (\tanh t)^{-\frac{d}{2} \left( 1- \frac{1}{q} \right)} \| f \|_{L^1}.
\end{align*}
This completes the proof.
\end{proof}

Note that Lemma \ref{dc} essentially gives the desired fixed-time estimate of Theorem \ref{mt} \eqref{P2} for $\beta=1$ -- see also Remark \ref{IrD} below. In order to deal with the case $0<\beta <1$, Bochner’s subordination formula and the property of probability density function (see \eqref{p2}) will play a crucial role. To be precise, Bochner’s subordination formula allows us to express the heat semigroup $e^{-t\sqrt{H}}$  in terms of solutions of the heat equation:
\begin{eqnarray}\label{ifh}
e^{-t\sqrt{H}}f(x)=\pi^{-1/2}\int_0^{\infty} e^{-y} \, e^{-\frac{t^2}{4y} H} f(x) \, y^{-1/2} \, dy, 
\end{eqnarray}
which ultimately follows from the identity
\[ e^{-a}=\pi^{-1/2}\int_0^{\infty} e^{-y} \, e^{-\frac{a^2}{4y}} \, y^{-1/2} \, dy \quad (a>0).\]
The Macdonald function $K_\nu(z)$ is defined, for $z > 0$, by
\[ K_{\nu}(z)=2^{-\nu-1} \, z^{\nu} \, \int_0^{\infty} e^{-y-\frac{z^2}{4y}} \, y^{-\nu-1} \, dy.\]
A straightforward change of variables shows that
\[ z^{\nu} K_{\nu}(z)=2^{\nu-1} \, \int_0^{\infty} e^{-y-\frac{z^2}{4y}} \, y^{\nu-1} \, dy=z^{\nu} K_{-\nu}(z). \]
Then $z^{\nu}K_\nu (z)$ converges to $2^{\nu-1} \Gamma(\nu)$ as $z \to 0$. Moreover,
it is known that $K_\nu (z)$ has exponential decay at infinity (see \cite{lebedevspecial}).
Consider now the Gaussian kernel of the form
\[ g_t(x)=(4\pi t)^{-d/2} \, e^{-\frac{|x|^2}{4t}},~t>0,~x\in \R^d.\]
We set $p_t(x,y)=p_t(x-y)$, where
\[ p_t(x)=\int_0^{\infty} g_s(x) \, \eta_t(s) \, ds,\]
$g_s$ is the Gaussian kernel defined above and $\eta_t\geq 0$ is the density function of the distribution of the $\beta$-stable subordinator at time $t$, see e.g., \cite{bogdan2009potential, bogdan2016hardy}. Therefore, $\eta_t(s)=0$ for $s\leq 0$ and, for $0<\beta<1$, we have
\begin{eqnarray}\label{p1}
\int_0^{\infty} e^{-us} \, \eta_t(s) \, ds=e^{-tu^{\beta}}, \quad  u\geq0.
\end{eqnarray}
The fractional heat semi group $e^{-tH^{\beta}}$ is thus given in terms of solutions of the heat equation:
\begin{eqnarray}\label{p2}
e^{-tH^{\beta}}f(x)=\int_0^{\infty} e^{-sH} f(x) \, \eta_t(s) \, ds.
\end{eqnarray}
We are now ready to complete the proof of Theorem \ref{mt}. 
\begin{proof}[\textbf{Proof of Theorem  \ref{mt} -- Part \eqref{P2}}]
The  case $t>1$ follows from the proof of Part \eqref{P1} of Theorem \ref{mt}, as it holds for all $p,q \in [1, \infty]$. We then assume $0<t\leq 1$ from now on. 
In view of the identity \eqref{p2} and Lemma \ref{dc} for the case $ \beta = 1 $, we obtain
\[ \| e^{-tH^\beta}f \|_{L^q} \leq C \left[\int_0^\infty   (\tanh s)^{-\alpha/2} \eta_t(s) ds\right]  \,  \|f\|_{L^p},\]
where we set $ \alpha = d| 1/p-1/q|$. Splitting the integral above into two parts, the integral taken over $ [1,\infty)$ is bounded by
\[  \int_0^\infty     \eta_t(s) ds  = 1.\]
The remaining  integral is bounded by 
\[  \int_0^\infty    s^{-\alpha/2} \eta_t(s) ds = \frac{1}{\Gamma(\alpha/2)}   \int_0^\infty  \Big( \int_0^\infty   e^{-us} u^{\alpha/2-1} du \Big)   \eta_t(s) ds.\]
Changing the order of integration, and using \eqref{p1}, for a suitable constant $C>0$ we obtain 
\[ \int_0^\infty    s^{-\alpha/2} \eta_t(s) ds \leq C \int_0^\infty u^{\alpha/2-1} e^{-t u^\beta} du.\]
Finally, the change of variables $ v = u^\beta $ gives the estimate 
\[ \int_0^\infty u^{\alpha/2-1} e^{-t u^\beta} du \leq C \int_0^\infty v^{(\alpha/2\beta) -1} e^{-tv} dv = C_{\alpha, \beta} t^{-(\alpha/2\beta)}.\]
This completes the proof for the case $0<t\leq 1$.
\end{proof}

\begin{Remark} \label{IrD}  We would like to have also a representation in the vein of \eqref{re} for the fractional heat propagator $e^{-tH^\beta}$ with $\beta >1$ in terms of the Weyl transform. On the other hand, we have a convolution formula for the classical fractional heat propagator $e^{-t(-\Delta)^\beta}$. Regretfully, we do not know how to get fixed-time estimates for $\beta> 1$ via the Weyl transform at the time. 
\end{Remark}

\begin{Remark}
Using the fact that $ e^{-tH} $ commutes with the Fourier transform, i.e., $ \widehat{e^{-tH}f} = e^{-tH}\hat{f},$ one obtains
\[
\|e^{-tH}f\|_{\mathcal{F}L^q} \leq C  (\tanh t)^{-\frac{d}{2} \left|\frac{1}{q} - \frac{1}{p}\right| } \,    \|f\|_{\mathcal{F}L^p},
\]
where the Fourier-Lebesgue spaces $\mathcal{F}L^p(\mathbb R^d)$ is defined by 
\[ \mathcal{F}L^p(\mathbb R^d)= \left\{f\in \mathcal{S}'(\mathbb R^d): \|f\|_{\mathcal{F}L^{p}}\coloneqq  \|\hat{f}\|_{L^{p}}< \infty \right\}.\]  
\end{Remark}

\section{Proof of Theorem \ref{LWL}}\label{AWT}
\subsection{Part \eqref{LWL1} -- local wellposedness}  Fix $M_1 \geq \|u_0\|_{L^p}.$
The proof strategy is quite standard. Let $T>0$ and set \[Y_T=L^{\infty}\left( (0,T), L^p(\R^d)\right) \cap L^{\infty}\left((0,T), L^{p \gamma}(\R^d)\right),\]
endowed with a norm
\[\|u\|_{Y_T}=\max\left\{ \sup_{0<t<T} \|u(t)\|_{L^p}, \sup_{0<t<T}t^{\frac{d( \gamma-1)}{2p \gamma \, \beta}} \|u(t)\|_{L^{p\gamma}}  \right\}. \]
Moreover, consider
\[B_{M+1}=\{u \in Y_{T}:\|u\|_{Y_T}\leq M+1\}\]
where $M>0$ is chosen in such a way that $\|e^{-tH^\beta}u_0\|_{Y_T} \leq C M_1 \leq M$. Note that $M$ depends only on $\|u_0\|_{Y_T}$ -- in particular, it is independent of $t$. 

Consider the mapping $\Phi \colon B_{M +1}  \to Y_T$ defined by
\begin{align*}
\Phi[u](t)=e^{-tH^\beta}u_0+\int_0^t e^{-(t-\tau)H^\beta}\left(|u(\tau)|^{\gamma-1} \, u(\tau) \right) \, d\tau.  
\end{align*}
We shall show that in fact $\Phi$ is a mapping from $B_{M +1}$ into $B_{M +1}$. Indeed, consider $u \in B_{M +1}$. By Theorem \ref{mt}, for $q \in \{ p, p\gamma\}$, we have 
\begin{align*}
\left\|\int_0^t e^{-(t-\tau)H^\beta}\left(|u(\tau)|^{\gamma-1} \, u(\tau) \right) \, d\tau\right\|_{L^q}
& \leq   C \int_0^t (t-\tau)^{-\frac{d}{2\beta}[\frac{1}{p}-\frac{1}{q}]} \,  \|u(\tau)\|_{L^{p\gamma}}^{\gamma}  d\tau\\
& \leq  C (M+1)^{\gamma}\int_0^t (t-\tau)^{-\frac{d}{2\beta}[\frac{1}{p}-\frac{1}{q}]} \,  \tau^{-\frac{d(\gamma-1)}{2p\beta}}  d\tau\\
&= \begin{multlined}[t] C (M+1)^{\gamma} \, t^{1-\frac{d}{2\beta}[\frac{1}{p}-\frac{1}{q}]-\frac{d(\gamma-1)}{2p\beta}} \\ \times \int_0^1 (1-\tau)^{-\frac{d}{2\beta}[\frac{1}{p}-\frac{1}{q}]} \,  \tau^{-\frac{d(\gamma-1)}{2p\beta}}  d\tau.\end{multlined}
\end{align*}
Since $q=p$ or $q = p\gamma,~\gamma>1$ and  $p>p_c^\beta,$ we have 
\[ \int_0^1 (1-\tau)^{-\frac{d}{2\beta}[\frac{1}{p}-\frac{1}{q}]} \,  \tau^{-\frac{d(\gamma-1)}{2p\beta}}  d\tau <\infty.\]
Therefore, we infer 
\begin{equation} \label{E2.2}
t^{\frac{d}{2\beta}[\frac{1}{p}-\frac{1}{q}]}\left\|\int_0^t e^{-(t-\tau)H^\beta}\left(|u(\tau)|^{\gamma-1} \, u(\tau) \right) \, d\tau\right\|_{L^q}\leq C (M+1)^{\gamma} \, T^{1-\frac{d(\gamma-1)}{2p\beta}}.
\end{equation}
If we take $q =p$ or $q = p\gamma$ in \eqref{E2.2}, then 
\begin{eqnarray*}
\left\|\int_0^t e^{-(t-\tau)H^\beta}\left(|u(\tau)|^{\gamma-1} \, u(\tau) \right) \, d\tau\right\|_{L^p}\leq C_1 (M+1)^{\gamma} \, T^{1-\frac{d(\gamma-1)}{2p\beta}}
\end{eqnarray*}
or
\begin{eqnarray*}
t^{\frac{d(\gamma-1)}{2p\gamma \beta}} \, \left\|\int_0^t e^{-(t-\tau)H^\beta}\left(|u(\tau)|^{\gamma-1} \, u(\tau) \right) \, d\tau\right\|_{L^{p\gamma}}\leq C_2 (M+1)^{\gamma} \, T^{1-\frac{d(\gamma-1)}{2p\beta}}.
\end{eqnarray*}
As a result, we conclude that \[\|\Phi[u]\|_{Y_T}\leq M+\max \{C_1,C_2\} \, (M+1)^{\gamma} \, T^{1-\frac{d({\gamma}-1)}{2p\beta}}. \]
Moreover, for a sufficiently small $T > 0$, we have
\[\max \{C_1,C_2\} \, (M+1)^{\gamma} \, T^{1-\frac{d(\gamma-1)}{2p\beta}} \leq 1.\]
This shows that $\Phi$ is a mapping from $B_{M +1}$ into $B_{M +1}$, as claimed.

We now prove that $\Phi \colon B_{M +1} \to Y_T$ is a contraction mapping. Recall that 
\begin{eqnarray}\label{ei}
\left| |u|^{\gamma-1}u- |v|^{\gamma-1}v\right| \lesssim_{\gamma} \left( |u|^{\gamma-1} + |v|^{\gamma-1} \right) |u-v|.
\end{eqnarray}
By \eqref{ei} and H\"older  inequality, we have 
\begin{eqnarray*} \label{E2.3}
\||u|^{\gamma-1}u-|v|^{\gamma-1}v\|_{L^p}\leq \gamma \left(\|u\|_{L^{p\gamma}}^{\gamma-1}+\|v\|_{L^{p\gamma}}^{\gamma-1}\right) \, \, \|u-v\|_{L^{p\gamma}}.
\end{eqnarray*} 
 In light of the previous computation, for $u, v \in B_{M +1}$ and $q\in \{p, p\gamma\}$ we thus have 
\begin{align} \label{E2.4}
\|\Phi[u](t)-\Phi[v](t)\|_{L^q}\ & \leq  \gamma  \int_0^t (t-\tau)^{-\frac{d}{2\beta} \left( \frac{1}{p} -\frac{1}{q} \right)}  \left(\|u(\tau)\|_{L^{p\gamma}}^{\gamma-1}+\|v (\tau)\|_{L^{p\gamma}}^{\gamma-1}\right) \, \, \|u(\tau)-v(\tau)\|_{L^{p\gamma}} d\tau \nonumber  \\
& \leq   C_3 (M+1)^{\gamma-1} \, t^{1-\frac{d}{2\beta}[\frac{1}{p}-\frac{1}{q}]-\frac{d( \gamma-1)}{2p\beta}} \|u-v\|_{Y_T}
\end{align}
for a constant $C_3 > 0$. By taking $q = p$ or $q = p\gamma$ in \eqref{E2.4}, we similarly obtain
\[
\|\Phi[u](t)-\Phi[v](t)\|_{Y_T}\leq C_4 (M+1)^{\gamma-1} \,  T^{1-\frac{d(\gamma-1)}{2p\beta}} \|u-v\|_{Y_T}
\]
for a constant $C_4 > 0$. Since $1-\frac{d(\gamma-1)}{2p\beta} > 0$, for a sufficiently small $T > 0$ we have
\[C_4 (M+1)^{\gamma-1} \,  T^{1-\frac{d(\gamma-1)}{2p\beta}} \leq \frac{1}{2}.\]
We have thus proved that the mapping $\Phi$ is the contraction mapping for a sufficiently small
$T$. By Banach fixed point theorem, there exists a unique fixed point $u$ of the mapping $\Phi$
in $B_{M +1}$ and, in light of Duhamel's principle, the latter is a solution of \eqref{fHTE}.

\begin{Remark} We shall also mention that the result of Part \eqref{LWL1} can be alternatively derived from the abstract theorem of Weissler \cite[Theorem 1]{Fred1}. To this aim, we  define  $K_t(u)= e^{-tH^{\beta}} (|u|^{\gamma-1}u)$. Then for $t>0,$ $K_t:L^p(\R^d) \to L^p(\R^d)$ is  locally Lipschitz and 
\begin{align*}
\|K_t(u)-K_t(v)\|_{L^p} & \lesssim   t^{-\frac{d}{2\beta} \left( \frac{\gamma}{p} -   \frac{1}{p}\right)} \| |u|^{\gamma-1}u - |v|^{\gamma-1}v\|_{L^{\frac{p}{\gamma}}}\\
& \lesssim  t^{-\frac{d}{2\beta} \left( \frac{\gamma}{p} -   \frac{1}{p}\right)}
\left( \|u\|_{L^p}^{\gamma-1} + \|v\|_{L^p}^{\gamma-1} \right) \|u-v\|_{L^p}\\
& \lesssim  t^{-\frac{d}{2\beta} \left( \frac{\gamma}{p} -   \frac{1}{p}\right)} M^{\gamma-1} \|u-v\|_{L^p},
\end{align*}
for $\|u\|_{L^p}\leq M$ and $\|v\|_{L^p} \leq M.$  Since $p> \frac{d(\gamma-1)}{2\beta}$,  we have  $t^{-\frac{d}{2\beta} \left( \frac{\gamma}{p} -   \frac{1}{p}\right)} \in L^1_{\mathrm{loc}}(0, \infty).$  Note that $t\mapsto \|K_t(0)\|_{L^p}=0 \in L^1_{\mathrm{loc}} (0, \infty)$ and  $e^{-sH}K_{t}= K_{t+s}$ for $t,s >0$. Then \eqref{LWL1} follows by \cite[Theorem 1]{Fred1}.
\end{Remark}

\noindent
\subsection{Part \eqref{LWL2} -- lower blow-up rate} Let $u_0 \in L^p(\R^d)$ be such that $T_{\max}<\infty$, and let $u\in C\left( [0, T_{\max}), L^p(\R^d) \right)$ be the maximal solution of \eqref{fHTE}. Fix $s\in [0, T_{\max})$ and let 
\[w(t)= u(t+s), \quad t\in [0, T_{\max}-s ), \]
with $w(0)=u(s).$  Then, as in the proof of Part \eqref{LWL1}, we claim that
\begin{eqnarray}\label{dfp}
\|u(s)\|_{L^p} + K M^{\gamma} (T_{\max}-s)^{1- \frac{d(\gamma-1)}{2p\beta}}> M, \quad \forall M>0,
\end{eqnarray}
for some constant $K>0$. If this were not the case, there would exist $M>0$ such that 
\[ \|u(s)\|_{L^p}+K M^{\gamma} (T_{\max}-s)^{1- \frac{d(\gamma-1)}{2p\beta}} \leq M, \]
and $w$ would be defined on $[0, T_{\max}-s]$ -- in particular, $u(T_{\max})$ would be well defined, a contradiction.  Hence, \eqref{dfp} is verified, for any $t\in [0, T_{\max})$ fixed and for all $M>0$. 

Set then $M= 2 \|u(t)\|_{L^p}$. By \eqref{dfp}, we infer 
\[ \|u(t)\|_{L^p} + K2^{\gamma} \|u(t)\|_{L^p}^{\gamma} \left( T_{\max}-t \right)^{1- \frac{d(\gamma-1)}{2p\beta}}> 2 \|u(t)\|_{L^p}, \quad \forall t\in [0, T_{\max}). \] Hence,  we have 
\[ \|u(t)\|_{L^p} \geq C  \left( T_{\max}- t \right) ^{\frac{d}{2p\beta}-\frac{1}{\gamma-1}}\quad \text{for all} \ t \in [0, T_{\max}). \]

\subsection{Part \eqref{T1.3} -- global existence} Given  $\gamma>1$, one can choose $r$ in such a way that 
\[\frac{2 \beta}{d \gamma (\gamma -1)} < \frac 1r < \frac{2 \beta}{d(\gamma-1)}.\]
Let $r$ be fixed once for all and set
\[\delta=\frac{1}{\gamma-1}-\frac{d}{2r \beta}.\]
We observe that 
\[\delta+1-\frac{d(\gamma-1)}{2r\beta}-\delta \gamma=0.\]
Suppose that $\rho > 0$ and $M > 0$ satisfy the inequality
\[\rho+KM^{\gamma} \leq M,\] where $K = K(\gamma, d, r) > 0$ is a constant and can explicitly be computed.
We claim that if 
\begin{equation}\label{t1}
\sup_{t>0} t^{\delta} \|e^{-tH^{\beta}}u_0\|_{L^r} \leq \rho
\end{equation}
then  there is a unique global solution $u$ of \eqref{fHTE} such that
\begin{equation}\label{c1}
\sup_{t>0} t^{\delta} \|u(t)\|_{L^{r}} \leq M.
\end{equation}
In order to prove our claim, consider 
\[X= \left\{ u\colon (0, \infty)\to L^{r}(\R^d):\sup_{t>0} t^{\delta} \|u(t)\|_{L^r}< \infty \right\},\]
\[X_{M}= \left\{ u\in X:  \sup_{t>0} t^{\delta} \|u(t)\|_{L^r} \leq M \right\},\quad d(u,v)= \sup_{t>0} t^{\delta} \|u(t) -v(t)\|_{L^r}.\] It is easy to realize that $(X_M,d)$ is a complete metric space.

Consider now the mapping
\begin{align}\label{dhumalterm}
 \mathcal{J}_{u_0}(u)(t)= e^{-tH^{\beta}} u_0+\int_0^t e^{- (t-s) H^{\beta}} (|u(s)|^{\gamma-1} u(s)) ds. 
 \end{align}
Let $u_0$ and $v_0$ satisfy \eqref{t1} and choose $u,v \in X_M$. Clearly, we have 
\begin{multline*} t^{\delta} \| \mathcal{J}_{u_0}u(t)- \mathcal{J}_{v_0}v(t)\|_{L^r} \leq t^{\delta} \|e^{-tH^{\beta}}(u_0-v_0)\|_{L^r} \\ + t^{\delta}\int_0^t\| e^{- (t-s) H^{\beta}} (|u(s)|^{\gamma-1} u(s)-|v(s)|^{\gamma-1} v(s))\|_{L^r} ds. \end{multline*}
Using Theorem \ref{mt} with exponents $(p, q)=(r/\gamma, r),$ \eqref{ei} and H\"older's inequality,  we obtain
\begin{align*}
\| e^{- (t-s) H^{\beta}} (|u(s)|^{\gamma-1} u(s)-&|v(s)|^{\gamma-1} v(s))\|_{L^r} \\ 
& \lesssim  (t-s)^{-\frac{d(\gamma-1)}{2r\beta}}  \||u(s)|^{\gamma-1} u(s)-|v(s)|^{\gamma-1} v(s)\|_{L^{\frac{r}{\gamma}}}\\
& \lesssim  (t-s)^{-\frac{d(\gamma-1)}{2r\beta}} \gamma \left( \|u(s)\|_{L^r}^{\gamma-1} + \|v(s)\|_{L^r}^{\gamma-1} \right) \|u(s)-v(s)\|_{L^r}\\
& \lesssim  (t-s)^{-\frac{d(\gamma-1)}{2r\beta}} \gamma s^{-\delta \gamma} M^{\gamma-1} d(u,v).
\end{align*}
Using this inequality,  we get
\begin{eqnarray} \label{Contraction}
t^{\delta} \| \mathcal{J}_{u_0}u(t)- \mathcal{J}_{v_0}v(t)\|_{L^r} &  \leq & t^{\delta} \|e^{-tH^{\beta}}(u_0-v_0)\|_{L^r} + t^{\delta} \gamma M^{\gamma-1} d(u,v )\int_0^t (t-s)^{-\frac{d(\gamma-1)}{2r\beta}} s^{-\delta \gamma} ds\nonumber\\
&\leq & t^{\delta} \|e^{-tH^{\beta}}(u_0-v_0)\|_{L^r} +K \,  M^{\gamma-1} d(u,v ), 
\end{eqnarray}
where $K=t^{\delta} \gamma \,  \int_0^t (t-s)^{-\frac{d(\gamma-1)}{2r\beta}} s^{-\delta \gamma} ds$ is a finite positive constant.
Indeed, since 
\[\delta \gamma <1,~~\frac{d(\gamma-1)}{2r\beta}< 1,\] we see that 
\[ \int_0^t (t-s)^{-\frac{d(\gamma-1)}{2r\beta}} s^{-\delta \gamma} ds=t^{1-\frac{d(\gamma-1)}{2r\beta}-\delta \gamma}\int_0^1 (1-s)^{-\frac{d(\gamma-1)}{2r\beta}} s^{-\delta \gamma} ds<\infty.\]
Setting $v_0=0$ and $v=0$ in \eqref{Contraction} we have
\[t^{\delta} \| \mathcal{J}_{u_0}u(t)\|_{L^r}\leq \rho+KM^{\gamma}\leq M.\]
That is, $\mathcal{J}_{u_0}$ maps $X_M$ into itself. Letting $u_0=v_0$ in \eqref{Contraction}, we note that
\[d (\mathcal{J}_{u_0}u(t), \mathcal{J}_{u_0}v(t)) \leq K \,  M^{\gamma-1} d(u,v ). \]
Since $KM^{\gamma-1} < 1$, then $\mathcal{J}_{u_0}$ is a strict contraction on $X_M$. Therefore, $\mathcal{J}_{u_0}$ has a unique fixed point $u$ in $X_M$, which is a solution of \eqref{dhumalterm}.

Finally,  using Theorem \ref{mt} with exponents $(p,q)=\left(\frac{d(\gamma-1)}{2\beta}, r \right)$,  we see that if $\|u_0\|_{L^{p_c^{\beta}}}$ is sufficiently small then \eqref{t1} is satisfied. 

\section{Concluding  remarks}\label{cr}
In this concluding section we illustrate another application of Theorem \ref{mt}, that is a set of Strichartz estimates for the fractional heat propagator. We emphasize that Strichartz estimates are indispensable tools for a thorough study of the wellposedness theory for nonlinear equations -- see e.g., \cite{tao2006nonlinear, wang2011harmonic}. Since the proof is based on a standard machinery, via $TT^{\star}$ method and  real  interpolation (see  for instance   \cite[Lemma 3.2]{miao2008well}  and \cite[Theorem 1.4]{zhai2009strichartz} and the references therein),  we omit the details.

We say that $(q, p,r)$ is an \textit{admissible triplet} of indices if $1\leq p\leq r\leq \infty,  \beta>0$ and \[\frac1q = \frac{d}{2\beta}\left( \frac1r-\frac{1}{p}\right).\]
\begin{theorem} \label{sl} Let $I=[0,T)$ with $0<T\leq \infty$ and $\beta >0.$
\begin{enumerate}
\item \label{sl3} Let $(q, p,r)$ be any admissible triplet and consider $ f \in L^r(\R^d)$. Then $e^{-tH^{\beta}}f \in L^q(I, L^p(\R^d)) \cap C_b(I, L^r(\R^d))$ and there exists a constant $C>0$ such that
\[\|e^{-tH^{\beta}} f\|_{L^q(I, L^p)} \le C \|f\|_{L^r}.\]
\item \label{sl4} Let $p_1',p \in (1,  \infty),$ or $(p_1',p)=(1, \infty),$ or $p_1'=1$ and $p \in [2, \infty),$ or  $p_1'\in (1, \infty)$ and $p=1.$  Assume that  $(q, p)$ and  $(q_1, p_1)$ satisfy $ p_1' \ne p, 1<q_1'<q< \infty$ and
\begin{eqnarray}\label{wad}
\frac{1}{q_1'} + \frac{d}{2\beta} \left| \frac{1}{p_1'} - \frac{1}{p} \right| = 1+ \frac{1}{q}.
\end{eqnarray} Then, there exists a constant $C>0$ such that
\begin{eqnarray*}\label{ihs}
\left\| \int_0^t e^{- (t-s) H^{\beta}} F(s) ds \right\|_{L^q(I, L^p(\R^d))} \le C \|F\|_{L^{q_1'}(I, L^{p_1'} (\R^d))}.
 \end{eqnarray*}
\end{enumerate}
\end{theorem}
We note that  Pierfelice \cite{pierfelice2006strichartz} studied  Strichartz estimates for  \eqref{fHTE} with $H= -\Delta$, while Miao-Yuan-Zhang \cite{miao2008well} and  Zhai \cite{zhai2009strichartz} obtained Strichartz estimates for the fractional Laplacian $(-\Delta)^{\beta}$.  

\begin{Remark} Taking Theorem \ref{mt} into account, 
part \eqref{sl3} of Theorem \ref{sl},  can be proved in analogy with \cite[Lemma 3.2]{miao2008well}  and part \eqref{sl4} of Theorem \ref{sl} can be proved in analogy with  \cite[Theorem 1.4]{zhai2009strichartz}.  The property \eqref{wad}
is weaker than  the admissibility of triplets  $(q, p,2)$ and $(q_1, p_1,2)$. The hypothesis \eqref{wad} and the constraint $ p_1' \ne p,~1<q_1'<q< \infty$ appear as a consequence of the Hardy-Littlewood-Sobolev inequality.  While the hypothesis on $(p_1',p)$ is needed in order to apply Theorem \ref{mt} when dealing with truncated decay estimates (see \cite[Section 3.2]{zhai2009strichartz}).
\end{Remark}
\section*{Acknowledgments} \noindent
 D.G.B. is thankful to DST-INSPIRE (DST/INSPIRE/04/2016/001507) for the research grant. \\ 
 \noindent R.M. acknowledges the support of DST-INSPIRE [DST/INSPIRE/04/2019/001914] for research grants.\\
\noindent F.N.\ and S.I.T. are members of the Gruppo Nazionale per l'Analisi Matematica, la Probabilità e le loro Applicazioni (GNAMPA) of the Istituto Nazionale di Alta Matematica (INdAM). \\ 
\noindent S.T. is supported by J.C.Bose Fellowship from DST, Government of India.


\end{document}